\documentclass[11pt,a4paper]{article}
\usepackage{epsf,epsfig,amsfonts,amsgen,amsmath,amssymb,amstext,amsbsy,amsopn,amsthm,lineno}
\usepackage{color}
\usepackage{graphicx}
\usepackage{caption}
\usepackage{authblk}

\newtheorem{theorem}{Theorem}

\newtheorem{prop}[theorem]{Proposition}
\newtheorem{lemma}[theorem]{Lemma}
\newtheorem{cor}[theorem]{Corollary}

\theoremstyle{definition}
\newtheorem{definition}{Definition}
\newtheorem{example}{Example}

\newtheorem{claim}{Claim}
\newtheorem{observation}[theorem]{Observation}

\newtheorem{conjecture}[theorem]{Conjecture}

\begin{document}
	\title
	{\bf\Large Vertex-disjoint properly edge-colored cycles in edge-colored complete graphs}
		\date{}
		\author[1,2]{\small Ruonan Li \thanks{Supported by CSC (No.~201506290097); E-mail: liruonan@mail.nwpu.edu.cn}}
		\author[2]{\small Hajo Broersma \thanks{Corresponding author; E-mail: h.j.broersma@utwente.nl}}
		\author[1]{\small Shenggui Zhang\thanks{Supported by NSFC (No.11671320); E-mail:
				sgzhang@nwpu.edu.cn}}
		\affil[1]{Department of Applied Mathematics, Northwestern Polytechnical University,	
			
			Xi'an, 710072, P.R.~China}
		\affil[2]{ Faculty of EEMCS, University of Twente, P.O. Box 217,
			
			7500 AE Enschede, The Netherlands
	}
	\maketitle
	\begin{abstract}
		\noindent
		It is conjectured that every edge-colored complete graph $G$ on $n$ vertices 
		satisfying $\Delta^{mon}(G)\leq n-3k+1$ contains $k$ vertex-disjoint properly edge-colored cycles. We confirm this conjecture for $k=2$, 
		prove several additional weaker results for general $k$, and we establish structural properties of possible minimum counterexamples to the conjecture. We also reveal a close relationship between properly edge-colored cycles in edge-colored complete graphs and directed cycles in multi-partite tournaments.
		Using this relationship and our results on edge-colored complete graphs, we obtain several partial solutions to a conjecture on disjoint cycles in directed graphs due to Bermond and Thomassen.
		\medskip
		
		\smallskip
		
	\end{abstract}
	\section{Introduction}
	All graphs considered in this paper are finite and simple. For terminology and notation not defined here, we refer the reader to \cite{Bondy:2008}.
	
	Let $G$ be a graph with vertex set $V(G)$ and edge set $E(G)$. For a nonempty subset $S$ of $V(G)$, let $G[S]$ denote the subgraph of $G$ induced by $S$, and let $G-S$ denote the subgraph of $G$ induced by $V(G)\setminus S$. When $S=\{v\}$, we use $G-v$ instead of $G-\{v\}$. An {\it edge-coloring\/} of $G$ is a mapping $col:E(G)\rightarrow \mathbb{N}$, where $\mathbb{N}$ is the set of natural numbers. A graph $G$ with an edge-coloring is called an {\it edge-colored\/} graph or simply a {\it colored\/} graph. We say that a colored graph $G$ is a {\it properly colored\/} graph or simply a {\it PC\/} graph if each pair of adjacent edges (i.e., edges incident with one common vertex) in $G$ are assigned distinct colors. A PC graph $G$ is called a {\it rainbow\/} graph if all the edges of $G$ are assigned different colors. 
	
	Let $G$ be a colored graph. For a vertex $v\in V(G)$, the {\it color degree\/} of $v$, denoted by $d^c_G(v)$ is the number of different colors appearing on the edges incident with $v$. For an edge  $e\in E(G)$, let $col_G(e)$ denote the color of $e$. For a subgraph $H$ of $G$, let $col_G(H)$ denote the set of colors appearing on $E(H)$. For two vertex-disjoint subgraphs $F$ and $H$ of $G$, let $col_G(F,H)$ denote the set of colors appearing on the edges between $F$ and $H$. If $V(F)=\{v\}$, then we often write $col_G(v,H)$ instead of $col_G(F,H)$. For two disjoint nonempty subsets $S$ and $T$ of $V(G)$, we use $col_G(S,T)$ as shorthand for $col_G(G[S],G[T])$. When there is no ambiguity, we often write $d^c(v)$ for $d_G^c(v)$, $col(e)$ for $col_G(e)$, $col(H)$ for $col_G(H)$, $col(F,H)$ for $col_G(F,H)$, $col(v,H)$ for $col_G(v,H)$ and $col(S,T)$ for $col_G(S,T)$. For each color $i\in col(G)$, we use $G^i$ to denote the spanning subgraph of $G$ induced by the edges of color $i$ in $G$. Let $\Delta^{mon}(G)$ denote the {\it maximum monochromatic degree\/} of $G$, i.e., $\Delta^{mon}(G)=\{\Delta(G^i): i\in col(G)\}$. Throughout this paper, we use $C_3$ and $C_4$ to denote cycles of length $3$ and $4$, respectively. We also frequently use the term triangle for a $C_3$.
	
	Research problems related to PC cycles and rainbow cycles in colored graphs have attracted a lot of attention during the past  decades, not only because of the many challenging open problems and conjectures and interesting results, but also because of the relation to problems on  directed cycles in directed graphs. {We refer the reader to \cite{X.Li: 2008} and Chapter 16 in \cite{Bang-Jensen-Gutin: 2009} for surveys on rainbow cycles and PC cycles, respectively. We also recommend proof techniques in \cite{B.Li: 2014} and \cite{Lo: 2014-2}, and constructions in \cite{Fujita: 2011} and Chapter 16 in \cite{Bang-Jensen-Gutin: 2009} for a glance of the deep relation between edge-colored graphs and directed graphs. Here, we are mainly interested in the existence of vertex-disjoint  PC cycles (called disjoint PC cycles for simplicity in the sequel) in colored complete graphs. Our first easy observation implies that having $k$ disjoint PC cycles is equivalent to having $k$ disjoint PC cycles of length at most $4$ in colored complete graphs.
	
	\begin{observation}\label{obs:1}
		Let $G$ be a colored $K_n$ and let $C$ be a PC cycle in $G$. Then, for each vertex $v\in V(C)$, there exists a PC cycle $C'$ of length at most $4$ in $G$ containing $v$ and with $V(C')\subseteq V(C)$.
	\end{observation}
	\begin{proof}
		By contradiction. Suppose that $C=v_1v_2\cdots v_rv_1$ is a PC cycle of minimal length in $G$ containing $v_1$, and assume that $r\geq 5$. Assume without loss of generality that $col(v_1v_2)=1$ and $col(v_1v_r)=2$. If $col(v_3v_4)=1$, then, using the chord $v_1v_4$, either $v_1v_2\cdots v_4v_1$ or $v_1v_rv_{r-1}\cdots v_4v_1$ is a PC cycle containing $v_1$ and shorter than $C$, a contradiction. So, $col(v_3v_4)\neq 1$. Similarly, using the chord $v_1v_3$, we can show that  $col(v_3v_4)\neq 2$. Without loss of generality, assume that $col(v_3v_4)=3$. If $col(v_1v_3)=3$, then $v_1v_2v_3v_1$ is a PC cycle containing $v_1$ and shorter than $C$, a contradiction. So, $col(v_1v_3)\neq 3$. Similarly, we can show that  $col(v_1v_4)\neq 3$. Since by the choice of $C$, $v_1v_2v_3v_4v_1$ and $v_1v_rv_{r-1}v_4v_3v_1$ are not PC cycles, we conclude that $col(v_1v_4)=1$ and $col(v_1v_3)=2$. This implies that $v_1v_3v_4v_1$ is a rainbow triangle, our final contradiction. 
	\end{proof}
	Before turning to disjoint PC cycles, we first recall the following fundamental result on the existence of PC cycles in colored graphs.
	
	\begin{theorem}[Grossman and {H\"{a}ggkvist} \cite{Grossman:1983}, Yeo \cite{Yeo: 1997}]\label{Yeo}
		Let $G$ be an edge-colored graph containing no PC cycles. Then $G$ contains a vertex $v$ such that no component of $G-v$ is joined to $v$ with edges of more than one color.
	\end{theorem}
	Combining Theorem~\ref{Yeo} and Observation~\ref{obs:1} for colored complete graphs, we immediately obtain a maximum monochromatic degree condition for the existence of a PC $C_3$ or $C_4$. 
	
	\begin{observation}\label{obs:2}
		Let $G$ be a colored $K_n$. If $\Delta^{mon}(G)\leq n-2$, then $G$ contains a PC cycle of length at most 4.
	\end{observation} 
	
	The observation follows from the simple fact that in a complete graph $G$ on at least two vertices, for every vertex $v$ of $G$, $G-v$ consists of only one component. 
	
	Using Observations~\ref{obs:1} and \ref{obs:2}, and repeatedly deleting the vertices of PC cycles of length at most 4, it is easy to obtain the following sufficient condition for the existence of $k$ disjoint PC cycles.
	\begin{observation}\label{obs:3}
		Let $G$ be a colored $K_n$. If $\Delta^{mon}(G)\leq n-4k+2$, then $G$ contains $k$ disjoint PC cycles of length at most 4.
	\end{observation}
	
	Motivated by the above observations, our aim is to find a (best possible) positive function $g(k)$ (only depending on $k$) such that every colored complete graph $G$ with $\Delta^{mon}(G)\leq n-g(k)$ 
	contains $k$ disjoint PC cycles. We conjecture that the following holds.
	
	\begin{conjecture}\label{con:disjoint}
		Let $G$ be a colored $K_n$. If $\Delta^{mon}(G)\leq n-3k+1$, then $G$ contains $k$ disjoint PC cycles of length at most 4.
	\end{conjecture}
	We confirm this conjecture for the case that $k=2$.
	\begin{theorem}\label{thm:k=2}
		Let $G$ be a colored $K_n$. If $\Delta^{mon}(G)\leq n-5$, then $G$ contains two disjoint PC cycles of length at most $4$.
	\end{theorem}
	We postpone the proof of Theorem~\ref{thm:k=2} to Section~\ref{sec:mainproof}. In the next section, we give several additional results related to Conjecture~\ref{con:disjoint}. The proofs of these results can be found in Sections~\ref{sec:otherproofs} and \ref{sec:eq}. 
	
	We continue here with some examples to discuss the tightness of the bounds in Conjecture \ref{con:disjoint}.
    First of all, note that for a PC complete graph $G$ on $n={3k-1}$ vertices, $\Delta^{mon}(G)=1\leq n-3k+2$, while it cannot have $k$ disjoint PC cycles. This implies that the upper bound on $\Delta^{mon}(G)$ in Conjecture~\ref{con:disjoint} would be best possible, in a weak sense: for each $k$, this provides only one example.  
    When $k=2$, except for a PC $K_5$, Example \ref{exam:k=2} below also implies the tightness of the bound $n-5$. 
    \begin{example}\label{exam:k=2}
    Let $G$ be a colored complete graph with $V(G)=\{v_1,v_2,\ldots, v_6\}$. Decompose $G-v_1$ into two Hamilton cycles and color them by $\alpha$ and $\beta$, respectively. Color the edge $v_1v_i$ with $c_i$ for $i\in[1,6]$. Then $\Delta^{mon}(G)=6-4=2$, but $G$ cannot contain two disjoint PC cycles.
    \end{example}
    
    For $k=2$, we have no other examples to support the tightness of the bound in Conjecture~\ref{con:disjoint}. For $k\ge 3$,  we cannot find other examples to support the bound in Conjecture~\ref{con:disjoint} except for a PC $K_{3k-1}$. It is not unlikely that the bound in Conjecture~\ref{con:disjoint} can be improved for large $n$. The next example shows that for arbitrarily large $n$, we can construct a colored complete graph $G$ on $n$ vertices with $\Delta^{mon}(G)=n-\frac{3}{2}k$, but containing at most $k-1$ disjoint PC cycles.

    \begin{example}\label{exam:large_n}
    Given integers $k\geq 2$ ($k$ is even) and $n\geq \frac{9}{2}k-3$, let $G_1\cong K_{3k-3}$ with $V(G_1)=\{v_i:1\leq i\leq 3k-3\}$. Decompose $G_1$ into $\frac{3}{2}k-2$ Hamilton cycles. Arbitrarily choose a direction for each Hamilton cycle. For all $i,j\in [1,3k-3]$ and $i\neq j$, color the edge $v_iv_j$ with a color $c_j$ if and only if $v_j$ is the successor of $v_i$ in one of the Hamilton cycles. Let $G_2\cong K_{n-3k+3}$ with $V(G_2)=\{u_i:1\leq i\leq n-3k+3\}$ and $col(G_2)=\{\alpha\}$.
    Let $G$ be an edge-colored $K_n$ constructed by joining $G_1$ and $G_2$ such that $col(v_iu_{n-3k+3})=\beta$ for all $i\in [1,3k-3]$ and $col(v_iu_j)=c_i$ for all $i,j$ with $1\leq i\leq 3k-3$ and $1\leq j\leq n-3k+2$. Then $\Delta^{mon}(G)=n-\frac{3}{2}k$, but $G$ contains at most $k-1$ disjoint PC cycles.
    \end{example}  
    Since cycles in edge-colored graphs are closely related to cycles in digraphs, here we naturally think of disjoint dicycles in tournaments. In fact, Bang-Jensen et al. \cite{Bang-Jensen: 2014} proved that for every $\epsilon>0$, when $k$ is large enough, every tournament with minimum out-degree at least $(\frac{3}{2}+\epsilon)k$ contains $k$ disjoint cycles. And the linear factor $\frac{3}{2}$ is better than the factor $2$ that was conjectured by Bermond and Thomassen \cite{Bermond-Thomassen: 1981} in digraphs. In the light of the close relationship between PC cycles in colored complete graphs and directed cycles in multi-partite tournaments that we are going to discuss later, this could serve as supporting evidence that maybe the bound in  Conjecture~\ref{con:disjoint} can be improved when $n$ is sufficiently large.

    \section{Additional results related to Conjecture \ref{con:disjoint}}
	For the case that $k\geq 3$, our first additional result implies the existence of $k$ disjoint PC cycles if there exists a vertex in $G$ that is not contained in any PC cycle.
	
	\begin{theorem}\label{thm:eitheror}
		Let $G$ be a colored $K_n$. If $\Delta^{mon}(G)\leq n-3k+1$, then either $G$ contains $k$ disjoint PC cycles of length at most $4$, or each vertex of $G$ is contained in a PC $C_3$ or $C_4$.
	\end{theorem}
	Under some specific conditions, the bound for $\Delta^{mon}(G)$ in Theorem \ref{thm:eitheror} can be improved to $n-2k$.
	\begin{theorem}\label{thm:hasGallai}
		Let $G$ be a colored $K_n$ satisfying $\Delta^{mon}(G)\leq n-2k$. If $G$ has a Gallai 
		partition\footnote{See Definition \ref{def:Gallai} and Lemma \ref{lem:Gallai} in Section~\ref{sec:02} for more information on Gallai partitions}, then either $G$ contains $k$ disjoint PC cycles of length at most $4$, or each vertex of $G$ is contained in a PC $C_3$ or $C_4$.
	\end{theorem}
	With the same upper bound on $\Delta^{mon}(G)$, we can prove the following closely related result.
	
	\begin{theorem}\label{thm:PCdi2k3}
		Let $G$ be a colored $K_n$ satisfying $\Delta^{mon}(G)\leq n-2k$. Then either $G$ contains $k$ disjoint PC cycles of length at most $4$, or each vertex of $G$ with color degree at most $3$ is contained in a PC $C_3$ or $C_4$.
	\end{theorem} 
	
	Using some transformation techniques that we are going to specify later, it turns out that the results of Theorems \ref{thm:eitheror}, \ref{thm:hasGallai} and \ref{thm:PCdi2k3} are closely related to a problem on disjoint directed cycles (dicycles for short) in multi-partite tournaments. 
	
	In 1981, Bermond and Thomassen posed the following conjecture on the existence of $k$ disjoint dicycles in directed graphs. Here, $\delta^+(D)$ denotes the minimum out-degree of the directed graph $D$.
	\begin{conjecture}[Bermond and Thomassen \cite{Bermond-Thomassen: 1981}]\label{con:disjoint dicycles}
		Let $D$ be a directed graph. If $\delta^+(D)\geq 2k-1$, then $D$ contains $k$ disjoint dicycles.
	\end{conjecture}
	This conjecture has been confirmed for tournaments \cite{Bang-Jensen: 2014} and for bipartite tournaments \cite{Bai-Li: 2015} (for other progress on this conjecture, we refer to the introductory sections in \cite{Bai-Li: 2015} and \cite{Bang-Jensen: 2014}). We can state an equivalent of Conjecture \ref {con:disjoint dicycles} in terms of disjoint PC cycles when $D$ is a multi-partite tournament, using the following theorem.
	
	\begin{theorem}\label{thm:equivalent}
		Let $f(k)\geq 2k-1$, $\ell\geq 2$ and $I\subseteq\{a: a\geq 3, a\in \mathbb{N}\}$. Then Proposition \ref{pro:dicycle} and Proposition \ref{pro:PCcycle} below are equivalent.
	\end{theorem}
	
	The propositions in the above theorem deal with (either true or false) statements on dicyles in multi-partite tournaments and PC cycles in colored complete graphs, respectively, as specified below.  	
	\begin{prop}\label{pro:dicycle}
		Let $MT$ be an $\ell$-partite tournament without dicycles of length $i$ for all $i\in I$. If $\delta^+(MT)\geq f(k)$, then $MT$ contains $k$ disjoint dicycles. 
	\end{prop}
	\begin{prop}\label{pro:PCcycle}
		Let $G$ be a colored $K_n$ without PC cycles of length $i$ for all $i\in I$. If $\Delta^{mon}(G)\leq n-f(k)-1$, then either $G$ contains $k$ disjoint PC cycles, or each vertex of $G$ with color degree at most $\ell$ is contained in some PC cycle. 
	\end{prop}
	By directly using Theorem \ref{thm:equivalent}, we immediately obtain the following three 
	corollaries\footnote{During the process of writing this paper, we became aware of the fact that Y. Bai and B. Li \cite{PC} have obtained Corollaries \ref{cor:01}, \ref{cor:02} and \ref{cor:03} in 2015 using techniques in directed graphs. This work is still in progress.} 
	corresponding to Theorems \ref{thm:eitheror}, \ref{thm:hasGallai} and \ref{thm:PCdi2k3}, respectively.
	\begin{cor}\label{cor:01}
		Let $D$ be a multi-partite tournament. If $\delta^+(D)\geq 3k-2$, then $D$ contains $k$ disjoint dicycles. 
	\end{cor}
	
	\begin{cor}\label{cor:02}	
		Let $D$ be a multi-partite tournament containing no directed triangles. If $\delta^+(D)\geq 2k-1$, then  $D$ contains $k$ disjoint dicycles.
	\end{cor}
	
	\begin{cor}\label{cor:03}
		Let $D$ be a $2$-partite or $3$-partite tournament. 
		If $\delta^+(D)\geq 2k-1$, then $D$ contains $k$ disjoint dicycles. 
	\end{cor}
	
	Finally, we present some structural properties of a possible minimum counterexample $(G,k)$ to Conjecture \ref{con:disjoint}. Here, a minimum counterexample $(G,k)$ satisfies that $k$ is as small as possible, and subject to this, $|V(G)|$ is as small as possible, and subject to this, $|col(G)|$ is as small as possible.
	\begin{theorem}\label{thm:chracterization}
		Let $(G,k)$ be a minimum counterexample to Conjecture \ref{con:disjoint}.
		Then the following statements hold.\\
		$(a)$ $k\geq 3$;\\
		$(b)$ $|col(G)|=2 \text{~or~} 3$;\\
		$(c)$ $G$ contains no rainbow triangle;\\
		$(d)$ $G$ contains no monochromatic edge-cut;\\
		$(e)$ for each set $S\subseteq V(G)$ with $|S|\leq k-1$ and each vertex $v\in V(G)\setminus S$, there exists a PC $C_4$ in $G-S$ containing $v$.	
	\end{theorem}	
	All the omitted proofs of the above results (except for the corollaries) can be found in Sections~\ref{sec:mainproof}, \ref{sec:otherproofs} and \ref{sec:eq}, but we start with some additional terminology and auxiliary lemmas in the next section.
	
	\section{Terminology and Lemmas}\label{sec:02}
	Let $G$ be a colored complete graph. A {\it partition\/} 
	of $G$ is a family of subsets $U_1,U_2,\ldots, U_q$ of $V(G)$ satisfying $\bigcup_{1\leq i\leq q} U_i=V(G)$ and $U_i\cap U_j=\emptyset$ for $1\leq i<j\leq q$ (In the proofs, we sometimes allow that $U_i$ is an empty set).
	 For each partition $U_1,U_2,\ldots, U_q$ of $G$ and a vertex $x\in V(G)$, we use $U_x$ to denote the unique set $U_i~(1\leq i\leq q)$ containing $x$.  The following type of partition plays a key role in some of the proofs that follow. In this definition, the sets $U_i$ are supposed to be non-empty.
	\begin{definition}\label{def:Gallai}
		Let $G$ be a colored $K_n$. A partition $U_1,U_2,\ldots, U_q$ of $G$ is called a {\it Gallai partition\/} if $q\geq 2$, $|\bigcup_{1\leq i<j\leq q} col(U_i,U_j)|\leq 2$ and $|col(U_i,U_j)|=1$ for $1\leq i<j\leq q$.
	\end{definition}
	The following result shows that Gallai partitions exist in colored complete graphs without a PC $C_3$.
	\begin{lemma}[Gallai \cite{Gallai: 1967}]\label{lem:Gallai}
		Let $G$ be a colored $K_n$ with $n\geq 2$. If $G$ contains no rainbow triangles, then $G$ has a Gallai partition.
	\end{lemma}
	The next two lemmas deal with the cases that a colored complete graph $G$ does and does not contain a monochromatic edge-cut, respectively. In the presence of a monochromatic edge-cut in $G$, the degree condition $\Delta^{mon}(G)\leq n-2k$ easily implies the existence of $k$ disjoint PC cycles of length $4$, as is stated in the following result.
	\begin{lemma}\label{lem:no monochromatic}
		Let $G$ be a colored $K_n$ satisfying $\Delta^{mon}(G)\leq n-2k$. If $G$ contains a monochromatic edge-cut, then $G$ contains $k$ disjoint PC cycles of length $4$.
	\end{lemma}
	\begin{proof}
		Suppose that $G$ contains a monochromatic edge-cut, and let $V_1,V_2$ be a partition  of $G$ with only one color (say $red$) appearing on the edges between $V_1$ and $V_2$. The condition $\Delta^{mon}(G)\leq n-2k$ implies that each vertex of $V_1$ is joined to at least $2k-1$ vertices of $V_1$ with edges of colors distinct from $red$. Using induction on $k$, it is straightforward to see that this implies that there are $k$ disjoint edges $x_1x'_1, x_2x'_2,\ldots,x_kx'_k$ in $G[V_1]$ with colors distinct from $red$. By symmetry, there are also $k$ disjoint edges $y_1y'_1, y_2y'_2,\ldots,y_ky'_k$ in  $G[V_2]$  with colors distinct from $red$. Thus, $\{x_iy_ix'_iy'_ix_i: 1\leq i\leq k\}$ is a set of $k$ disjoint PC cycles of length 4.	
	\end{proof}
	In the absence of monochromatic edge-cuts in a colored complete graph $G$, we can use the following structural result for our proofs.
	\begin{lemma}\label{lem:seperation}
		Let $G$ be a colored $K_n$ ($n\geq 2$) without any monochromatic edge-cut. If there exists a vertex $v_0\in V(G)$ that is not contained in any PC cycles of length at most $4$ in $G$, then $G$ admits a partition $V_0,V_1,\ldots,V_p$ with\\
		$(a)$ $v_0\in V_0$, $2\leq p\leq d^c(v_0)$ and $|V_i|\geq 1$ for $0\leq i\leq p$;\\
		$(b)$ $col(V_0,V_i)=\{c_i\}$ for $1\leq i\leq p$ and $c_i\neq c_j$ for $1\leq i<j\leq p$;\\
		$(c)$ $col(V_i,V_j)\subseteq\{c_i,c_j\}$ for $1\leq i<j\leq p$;\\
		$(d)$ $col(G[V_i])=\{c_i\}$ for $1\leq i\leq p$.\\
		In particular, if $G$ has a Gallai partition, then there is a choice of $V_0,V_1,V_2,\ldots, V_p$ with $p=2$ and  a PC cycle $xyzwx$ with $x,z\in V_1$ and $y,w\in V_2$; if $G$ has no Gallai partition, then there exists a rainbow triangle $xyzx$ such that $V_x$, $V_y$ and $V_z$ are three distinct sets with $V_0\not\in \{V_x,V_y,V_z\}$. 
	\end{lemma}
	\begin{proof}
		Suppose that $G$ contains no monochromatic edge-cut, and let $v_0$ be a vertex of $G$ that is not contained in any PC cycle of length at most $4$. Then, we have $d^c(v_0)\geq 2$. Let $N^c(v_0)=\{c_1,c_2,\ldots,c_{d^c(v_0)}\}$ and let $S_i=\{v\in V(G): col(vv_0)=c_i\}$ for $1\leq i\leq d^c(v_0)$. Since $v_0$ is not contained in any PC triangle, we have $col(S_i,S_j)\subseteq \{c_i,c_j\}$. Thus, the sets $\{v_0 \}, S_1,S_2,\ldots, S_{d^c(v_0)}$ form a partition of $G$ satisfying $(a)$, $(b)$ and $(c)$. 
		
        Let $V_0,V_1,\ldots,V_p$ be a partition of $G$ satisfying $(a)$, $(b)$ and $(c)$, and with $|V_0|$ as large as possible. We will prove that this partition also satisfies $(d)$. Suppose it does not. Then, without loss of generality, assume that there exist vertices $x,y\in V_1$ such that $col(xy)\neq c_1$. For each vertex $v_j\in V_j~(2\leq j\leq p)$, on one hand, by $(c)$, we have $col(xv_j)\in \{c_1,c_j\}$; on the other hand, since $v_0yxv_jv_0$ is not a PC cycle, we have $col(xv_j)\in\{col(xy),c_j\}$. This forces that $col(xv_j)=c_j$. Similarly, we can prove that $col(yv_j)=c_j$. This implies that $col(x,V_j)=col(y,V_j)=\{c_j\}$ for all $j$ with $2\leq j\leq p$. Now define
        ~$$T_1=\{x\in V_1: \exists y\in V_1~s.t.~col(xy)\neq c_1\}.$$
		Then, $col(T_1,V_j)=\{c_j\}$ for $2\leq j\leq p$. Let	
		$V'_1=V_1\setminus T_1$ and $V'_0=V_0\cup T_1$. Then, $V'_0,V'_1,V_2,\ldots,V_p$ is a new partition of $G$. 
		If $V'_1\neq\emptyset$, then by the definition of $T_1$, we have $col(V'_1,T_1)=\{c_1\}$. Thus, $V'_0,V'_1,V_2,\ldots,V_p$ is a partition of $G$ satisfying $(a)$, $(b)$ and $(c)$ with $|V'_0|>|V_0|$. This contradicts the choice of $V_0,V_1,\ldots,V_p$.
		If $V'_1=\emptyset$, then $p\geq 3$ (otherwise, the edges between $V'_0$ and $V_2$ form a monochromatic edge-cut of $G$). Thus, $V'_0,V_2,\ldots,V_p$ is a partition of $G$ satisfying $(a)$, $(b)$ and $(c)$ with $|V'_0|>|V_0|$, a contradiction.
		
		If $G$ has a Gallai partition, then choose $U_0,U_1,\ldots,U_q$ as a Gallai partition of  $G$ with $v_0\in U_0$. Assume that the two colors between the sets are $red$ and $blue$. Since there is no monochromatic edge-cut in $G$, for each $i\in [0,q]$, there exist $s,t\in [0,q]$ with $i\neq s,i\neq t$ and $s\neq t$ such that $col(U_i,U_s)=\{red\}$ and $col(U_i,U_t)=\{blue\}$. Let $G'=G-U_0\setminus\{v_0\}$. Thus, $d^c_{G'}(v_0)=2$ and $v_0$ is not contained in any PC cycle in $G'$. If $G'$ contains a monochromatic edge-cut separating $S_1$ and $G'-S_1$, then the edges between $S_1\cup U_0$ and $G'-S_1$ form a monochromatic edge-cut of $G$. So, $G'$ does not contain a monochromatic edge-cut. Hence, $G'$ has a partition $V'_0,V'_1,V'_2$ satisfying $(a)$, $(b)$, $(c)$ and $(d)$. Let $V_0=V'_0\cup U_0$, $V_1=V'_1$ and $V_2=V'_2$. Then, $V_0,V_1,V_2$ is a partition of $G$ satisfying $(a)$, $(b)$, $(c)$ and $(d)$. In this case, we are left to prove the existence of a specific PC $C_4$.		
		If $G[V_1\cup V_2]$ contains a PC cycle, then it must be a PC $C_4$ (because $|col(G[V_1\cup V_2])|=2$) with two vertices in $V_1$ (say $x,z$) and two vertices in $V_2$ (say $y,w$). Since $col(xz)\neq col(yw)$, this $C_4$ must be $xyzwx$. So, it is sufficient to prove that $G[V_1\cup V_2]$ contains a PC cycle. Suppose the contrary. Then, by Theorem \ref{Yeo}, we may assume that there exists a vertex $x\in V_1$ joined to all the other vertices in $V_1\cup V_2$ with edges of the same color. If this unique color is $c_1$, then $d^c_G(x)=1$ and all the edges incident with $x$ form  a monochromatic edge-cut of $G$, a contradiction; otherwise, this unique color is $c_2$. This forces that $V_1=\{x\}$. Then, the edges between $V_0\cup \{x\}$ and $V_2$ form a monochromatic edge-cut of $G$, again a contradiction.

		If $G$ contains no Gallai partition, then by Lemma \ref{lem:Gallai}, $G$ must contain a rainbow triangle.  Let $V_0,V_1,\ldots,V_p$ be a partition of $G$ satisfying $(a)$, $(b)$, $(c)$ and $(d)$. Let $G'=G-V_0\setminus\{v_0\}$.  Then, $\{v_0\},V_1,\ldots,V_p$ is a partition of $G'$ satisfying $(a)$, $(b)$, $(c)$ and $(d)$. Now we are left to prove the existence of a specific rainbow triangle in $G$. Assume that $G'$ contains a rainbow triangle $xyzx$. Since $v_0$ is not contained in any PC cycles and $|col(G[V_i\cup V_j])|\leq 2$ for $1\leq i<j\leq p$, the vertices $x,y,z$ must come from different sets in  $\{V_1,\ldots,V_p\}$. So it is sufficient to prove that $G'$ contains a rainbow triangle. Suppose the contrary. Since $|V(G')|\geq p+1\geq 2$, by Lemma \ref{lem:Gallai}, $G'$ has a Gallai partition $U_0,U_1,U_2,\ldots, U_q~(q\geq 1)$ with $v_0\in U_0$. Thus, 
		$U_0\cup V_0,U_1,U_2,\ldots, U_q$ is a Gallai partition of $G$, a contradiction.
		
		This completes the proof.
	\end{proof}
	
	We now have all the necessary ingredients to prove our main theorem and the additional results.
	In the next section, we present our proof of Theorem \ref{thm:k=2}.
	
	\section{Proof of Theorem \ref{thm:k=2}}\label{sec:mainproof}
	\begin{proof}
		By contradiction. Let $G$ be a colored complete graph satisfying $\Delta^{mon}(G)\leq n-5$ but containing no two disjoint PC cycles. Since $\Delta^{mon}(G)\geq 1$, we have $n\geq 6$. If $G$ contains a rainbow triangle $uvwu$, then by deleting vertices $u,v$ and $w$ from $G$, we obtain a graph $G^\prime$ with $|V(G^\prime)|=n-3\geq 3$ and $\Delta^{mon}(G^\prime)\leq n-5=(n-3)-2$. So, by Observation~\ref{obs:2}, $G^\prime$ contains a PC cycle $C$ of length $3$ or $4$. Thus, the cycles $uvwu$ and $C$ form two disjoint PC cycles of length at most $4$, a contradiction. Hence $G$ contains no rainbow triangles, and, by Lemma~\ref{lem:Gallai}, $G$ has a Gallai partition. Let $U_1,U_2,\ldots, U_q$ be a Gallai partition of $G$ with $q$ as small as possible. By Lemma \ref{lem:no monochromatic}, $G$ contains no monochromatic edge-cut. So, we have $q\geq 4$. Assume that the two colors appearing between $U_i$ and $U_j~(1\leq i<j\leq q)$ are $red$ and $blue$.
		
		We proceed by proving six claims.
		
		\begin{claim}\label{clm:special_C4}
			There exists a PC $C_4$ in $G$ with vertices from distinct sets of $U_1,U_2,\ldots, U_q$.
		\end{claim}
		\begin{proof}
			Construct an auxiliary colored complete graph $H$ with $V(H)=\{x_1,x_2,\ldots,x_q\}$ and for $1\leq i<j\leq q$, color the edge $x_ix_j$ with the color that appears on the edges between $U_i$ and $U_j$.
			Since $G$ contains no monochromatic edge-cut, $col(H)=\{red,blue\}$ and $col(x,H-x)=\{red,blue\}$ for each vertex $x\in V(H)$. Thus, by Observation \ref{obs:2}, $H$ contains a PC $C_4$, which corresponds to a PC $C_4$ in $G$ with vertices from different sets of $U_1,U_2,\ldots, U_q$.
		\end{proof}
		
		Without loss of generality, assume that the PC $C_4$ in Claim \ref{clm:special_C4} is $C^*=v_1v_2v_3v_4v_1$ with $v_i\in U_i$ for $1\leq i\leq 4$, and satisfying that $col(v_1v_2)=col(v_3v_4)=red$ and $col(v_2v_3)=col(v_1v_4)=blue$ (see Figure \ref{fig:1-1}). Let $G^\prime= G-V(C^*)$. Since $|V(G')|\geq n-4\geq 2$, $G^\prime$ is nonempty. If $\Delta^{mon}(G^\prime)\leq |V(G')|-2$, then, by Observation~\ref{obs:2},  $G^\prime$ contains a PC $C_4$. Combining this PC cycle with $v_1v_2v_3v_4v_1$, we get two disjoint PC cycles of length $4$, a contradiction. So, there exists a vertex $v\in V(G^\prime)$ with $d^c_{G^\prime}(v)=1$ (see Figure \ref{fig:1-1}). Define 
		$$S_1=\{v\in V(G'): d^c_{G^\prime}(v)=1\}.$$
		Clearly, there is only one color in $col(S_1,G'-S_1)\cup col(G[S_1])$. We assert that this color must be $red$ or $blue$. Suppose not. Then, by the definition of Gallai partition, $V(G')$ is a subset of $U_i$ for some $i$ with $1\leq i \leq q$. Let $v_j$ be a vertex in $U_j$ for some $j$ with $1\leq j\leq q$ and $j\neq i$. Then,  the unique color in $col(U_i,U_j)$ appears at least $|V(G^\prime)|=n-4$ times at $v_j$. This contradicts that  $\Delta^{mon}(G)\leq n-5$.  Now, without loss of generality, assume that $col(S_1,G'-S_1)\cup col(G[S_1])=\{red\}$.
		\begin{figure}[h]
			\centering
			\includegraphics[width=0.45\linewidth]{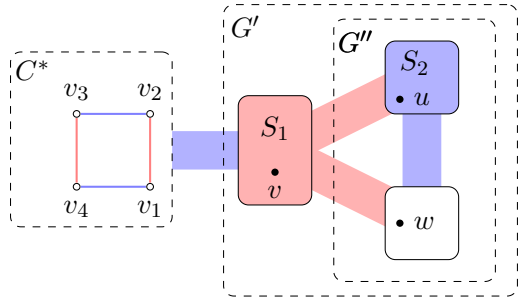}
			\caption{The coloring of $G$}\label{fig:1-1}
		\end{figure}
		\begin{claim}\label{clm:special_u}
			For each vertex $v\in S_1$, $v\not\in \cup_{1\leq i\leq 4} U_i$ and $col(v,C^*)=\{blue\}$.
		\end{claim}
		\begin{proof}
			Let $v$ be an arbitrary  vertex of $S_1$. By the assumption that $col(S_1,G'-S_1)\cup col(G[S_1])=\{red\}$, we have $col(v,G'-v)=\{red\}$. Then, $col(vv_i)\neq red$ for $1\leq i\leq 4$ (otherwise, the color $red$ would appear more than $n-5$ times at $v$, a contradiction). We further assert that $v\not\in \cup_{1\leq i\leq 4} U_i$. Suppose this is not the case. Then, $v\in U_i$ for some $1\leq i\leq 4$, and  $col(vv_{i+1})=red$ or $col(vv_{i-1})=red$ (where the indices are taken module $4$), a contradiction. This implies that $v\not\in \cup_{1\leq i\leq 4} U_i$ and $col(v,C^*)=\{blue\}$ (see Figure \ref{fig:1-1}). 
		\end{proof}
		
		\begin{claim}\label{clm:singleset}
			$U_i=\{v_i\}$ for $1\leq i\leq 4$.
		\end{claim}
		\begin{proof}
			Claim \ref{clm:special_u} shows that $S_1\cap \{\cup_{1\leq i\leq 4} U_i\}=\emptyset$. We are left to prove that $u\not\in \cup_{1\leq i\leq 4} U_i$ for each vertex $u\in V(G^\prime)\setminus S_1$. Note that for each vertex $u\in V(G^\prime)\setminus S_1$ and any vertex $v\in S_1$, we have $col(vu)=red\not\in col(v,C^*)$. This implies that $u\not\in \cup_{1\leq i\leq 4} U_i$.
		\end{proof}
		
		Now, for convenience, we call a cycle {\it special\/} if it is a PC cycle and its vertices come from different sets of $U_1,U_2,\ldots, U_q$.  We say a vertex $z\in V(G)\setminus V(C)$ is a {\it companion vertex\/} 
		of a special cycle $C$ if $z$ is joined to $C$ with color $blue$ ($red$) and joined to other vertices with color $red$ ($blue$). By Claims \ref{clm:special_u} and \ref{clm:singleset}, we know that\\
		$(a)$ each special cycle of length $4$ in $G$ has a companion vertex; \\
		$(b)$ if a vertex $v_i\in U_i$ is contained in a special cycle, then $U_i=\{v_i\}$. 
		\begin{claim}\label{clm:S_1}
			$|S_1|\leq 3$, and for each vertex $v_i$ ($1\leq i\leq 4$), there exist two distinct vertices $x_i,y_i\in V(G)\setminus (V(C^*)\cup S_1)$ such that $col(v_ix_i)=col(v_iy_i)=red$.
		\end{claim}
		\begin{proof}
			Suppose that $|S_1|\geq 4$. Let $x,y,z,w$ be four distinct vertices in $S_1$. Then, $xyv_1v_2x$ and $zwv_3v_4$ are two disjoint  PC cycles, a contradiction. So, we have $|S_1|\leq 3$. For each $i$ with $1\leq i\leq 4$, by Claim \ref{clm:singleset} and the definition of Gallai partition, we know that all the edges incident with $v_i$ are colored in $red$ or $blue$. Since each color appears at most $n-5$ times at $v_i$, we know that both $red$ and $blue$ appear at least 4 times at $v_i$. Note that  there are at most 2 vertices in $V(C^*)\cup S_1$ joined to $v_i$ by an edge with color $red$. Hence, there exist another two vertices $x_i,y_i\in V(G)\setminus (V(C^*)\cup S_1)$ joined to $v_i$ by an edge with color $red$.
		\end{proof}
		Let $G''=G'-S_1$. By Claim \ref{clm:S_1}, we have $|V(G'')|\geq 2$. If $\Delta^{mon}(G'')\leq |V(G'')|-2$, then, by Observation~\ref{obs:2}, $G''$ contains a PC $C_4$. Combining this cycle with $C^*$, we obtain two disjoint PC $C_4$s, a contradiction. Thus, there exists a vertex $u\in V(G'')$ such that $d^c_{G''}(u)=1$ (see Figure \ref{fig:1-1}). Define 
		$$S_2=\{u\in V(G''): d^c_{G''}(u)=1\}.$$
		\begin{claim}\label{clm:S_2}
			For each vertex $u\in S_2$, 
			$col(u,G''-u)=\{blue\}$, $U_u\cap S_1=\emptyset$, and $col(u,C^*)=\{red\}$.			
		\end{claim}
		\begin{proof}
			Let $u$ be a vertex in $S_2$. Then $d^c_{G''}(u)=1$.
			
			Suppose that $col(u,G''-u)=\{red\}$. Then, $col(u,G'-u)=col(u,G''-u)\cup col(u,S_1)=\{red\}$. This implies that $u\in S_1$, a contradiction.
			Suppose that the unique color in $col(u,G''-u)$ is neither $red$ nor $blue$. Then, by Claim \ref{clm:singleset} and the definition of Gallai partition, we have 
			$V(G'')\subseteq U_j$ for some $j$ with $5\leq j\leq q$. By Claim \ref{clm:S_1}, there are vertices $x_1,x_2,x_3,x_4$ in $V(G'')$ such that $col(v_ix_i)=red$ for all $i$ with $1\leq i\leq 4$. This implies that $col(U_i,U_j)=\{red\}$ for all $i$ with $1\leq i\leq 4$. Let 
			$$T_1=\{v_1,v_2,v_3,v_4\}, T_2=S_1, T_3=V(G'').$$
			It is easy to check that $col(T_1,T_2)=\{blue\}$, $col(T_2,T_3)=\{red\}$ and $col(T_1,T_3)=\{red\}$. Thus, the edges between $T_1\cup T_2$ and $ T_3$ form a $red$ edge-cut of $G$,  a contradiction. So, we have $col(u,G''-u)=\{blue\}$.
			
			Suppose that there exists a vertex $v\in S_1$ such that $v\in U_u$. Then, $u\in U_v$ and $col(u,C^*)=col(v,C^*)=\{blue\}$. Thus, $col(u, G-S_1)=col(u, G''-u)\cup col(u,C^*)=\{blue\}$. This implies that the color $blue$ appears at least $n-1-|S_1|\geq n-4$ times at the vertex $u$, a contradiction. Thus, we have $U_u\cap S_1=\emptyset$.
			
			\begin{figure}[h]
				\centering
				\includegraphics[width=0.45\linewidth]{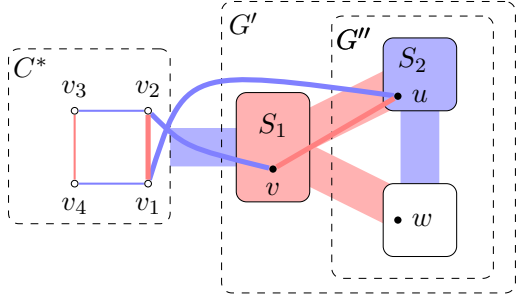}
				\caption{$col(uv_1)=blue$}\label{fig:2-1}
			\end{figure}
			Now we need to prove that $col(uv_i)=\{red\}$ for all $i$ with $1\leq i\leq 4$.  Suppose, to the contrary, that there exists a vertex (say $v_1$) on $C^*$ such that $col(uv_1)=blue$. Then, choose a vertex $v\in S_1$. Thus, the cycle $C=vuv_1v_2v$ is a special $C_4$ (see Figure \ref{fig:2-1}). Recall that each special cycle of length 4 has a companion vertex. Let $z$ be a companion vertex of $C$. For vertices $x\in V(G'')-u$ and 
			$y\in S_1-v$, we have $col(xu)\neq col(xv)$ and $col(yv_1)\neq col(yv)$. This implies that $z\not \in V(G'')\cup S_1$. Thus, $z$ is either $v_3$ or $v_4$. If $z=v_3$, then $col(z,C)=\{col(v_3v_2)\}=\{blue\}$. By the definition of $z$, we know that $col(z, V(G')-u-v)=\{red\}$. Note that $col(z,S_1)=col(v_3,S_1)=\{blue\}$. This forces that $S_1=\{v\}$. Now, for each vertex $x\in V(G)\backslash \{v_2,v_4,v,u\}$, we have $col(ux)=blue$. The color $blue$ appears at least $n-4$ times at $u$, a contradiction. So $u\neq v_3$. Similarly, we can prove that $u\neq v_4$. Thus, there is no choice for $z$, a contradiction. This implies that $col(uv_i)=red$ for all $1\leq i\leq 4$.
		\end{proof}
		\begin{claim}
			$V(G'')\setminus S_2\neq \emptyset$, and there exists a vertex $w\in V(G'')\setminus S_2$ such that $col(w,C^*)=\{red, blue\}$ and $w\not \in U_v\cup U_u$ for any vertices $v\in S_1$ and $u\in S_2$. 
		\end{claim}
		\begin{proof}
			If $V(G'')\setminus S_2=\emptyset$, then the edges between $S_2$ and $G-S_2$ form a $red$ edge-cut of $G$, a contradiction. So, we have $V(G'')\setminus S_2\neq \emptyset$. Suppose that for each vertex $w\in V(G'')\setminus S_2$, we have $col(wv_i)=col(wv_j)$ for all $1\leq i<j \leq 4$. Recall that $col(S_1,C^*)=\{blue\}$ and $col(S_2,C^*)=\{red\}$. We have $col(xv_i)=col(xv_j)$ for each vertex $x\in V(G)\setminus V(C^*)$ and $1\le i<j\leq 4$. Thus, 
			$$\{v_1,v_2,v_3,v_4\}, U_5, U_6,\ldots, U_q$$
			is also a  Gallai partition of $G$. This contradicts that $q$ is as small as possible. Thus, we can choose a vertex $w\in V(G'')\setminus S_2$ such that $col(w,C^*)=\{red, blue\}$. Since $col(S_1,C^*)=\{blue\}$ and $col(S_2,C^*)=\{red\}$, by the definition of Gallai partition, $w\not \in U_v\cup U_u$ for any vertices $v\in S_1$ and $u\in S_2$. 
		\end{proof}
		\begin{figure}[h]
			\centering
			\includegraphics[width=0.45\linewidth]{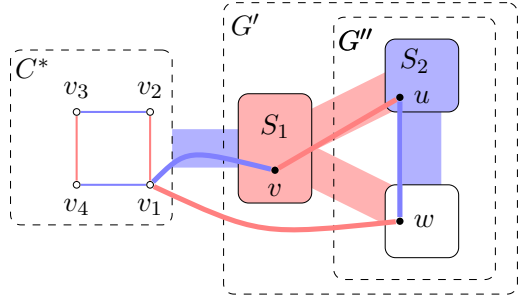}
			\caption{$col(wv_1)=red$}\label{fig:3-1}
		\end{figure}
		Since $col(w,C^*)=\{red, blue\}$, without loss of generality, assume that $col(wv_1)=red$. Choose vertices $v\in S_1$ and $u\in S_2$. Then, the cycle $C=vuwv_1v$ is a special cycle of length $4$ (see Figure \ref{fig:3-1}). Let $z$ be a companion vertex of $C$. Since $col(vx)\neq col(ux)$ for each vertex $x$ in  $G-S_1-u$, we have $z\in S_1-v$. However, for each vertex $y\in S_1-v$, we have $col(yv_1)=blue$ and $col(yu)=red$. Thus, $z\not\in S_1-v$. So there is no choice for $z$, a contradiction. This completes the proof of Theorem \ref{thm:k=2}.	
	\end{proof}
	
	\section{Proofs of Theorems \ref{thm:eitheror}, \ref{thm:hasGallai}, \ref{thm:PCdi2k3}, and  \ref {thm:chracterization}}\label{sec:otherproofs}
	By Observation \ref{obs:1}, the existence of $k$ disjoint PC cycles is equivalent to the existence of $k$ disjoint PC $C_3$s or $C_4$s. In this section, for convenience, we also use the term {\it short PC cycle(s)\/} instead of PC cycle(s) of length at most $4$.
	
	~\\
	{\bf Proof of Theorem \ref{thm:eitheror}}

	By contradiction. Let $G$ be a colored $K_n$. We say ($G,k$) is a counterexample to Theorem \ref{thm:eitheror} if $\Delta^{mon}(G)\leq n-3k+1$, but there are no $k$ disjoint short PC cycles in $G$ and not every vertex of $G$ is contained in a short PC cycle. 
	Let ($G,k$) be a counterexample to Theorem \ref{thm:eitheror} with $k$ as small as possible. By Observation \ref{obs:2} and Theorem \ref{thm:k=2}, we know that $k\geq 3$. If $G$ contains a rainbow triangle $xyzx$, then let $H=G - \{x,y,z\}$. Then, $\Delta^{mon}(H)\leq \Delta^{mon}(G)=n-3-3(k-1)+1$. Hence, by the choice of $(G,k)$, $H$ either contains $k-1$ disjoint short PC cycles, or each vertex of $H$ is contained in a short PC cycle. This in turn implies that either $G$ contains $k$ disjoint short PC cycles, or each vertex of $G$ is contained in a short PC cycle, a contradiction. Thus, $G$ contains no rainbow triangles and, due to Lemma~\ref{lem:Gallai}, has a Gallai partition. Note that $\Delta^{mon}(G)\leq n-3k+1< n-2k$. By Theorem \ref{thm:hasGallai}, $G$ either contains $k$ disjoint short PC cycles, or each vertex is contained in a short PC cycle. This  completes the proof.\qed
	
	~\\
	{\bf Proof of Theorem \ref{thm:hasGallai}}
	
	By contradiction. Let ($G,k$) be a counterexample to Theorem \ref{thm:hasGallai} with $k$ as small as possible. By Observation \ref{obs:2}, we have $k\geq 2$. By Lemma \ref{lem:no monochromatic}, $G$ contains no monochromatic edge-cut. Let $v_0$ be a vertex in $V(G)$ such that $v_0$ is not contained in any short PC cycle. Since $G$ contains a Gallai partition, by Lemma \ref{lem:seperation}, $V(G)$ can be separated into three non-empty sets $V_0,V_1,V_2$ (see Figure \ref{fig:4-1}) with
	$$v_0\in V_0,~col(V_0,V_1)=\{c_1\},~col(V_0,V_2)=\{c_2\},$$ $$col(V_1,V_2)\subseteq
	\{c_1,c_2\},~col(G[V_1])=\{c_1\},~ col(G[V_2
	])=\{c_2\}$$ 
	and $G$ contains a PC cycle $xyzwx$ with 
	$$x,z\in V_1,~y,w \in V_2.$$
	This implies that 
	$$|V_1|\geq 2 \text{~and~} |V_2|\geq 2.$$	
	\begin{figure}[h]
		\centering
		\includegraphics[width=0.28\linewidth]{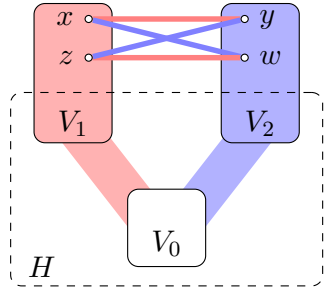}
		\caption{$G$ in the proof of Theorem \ref{thm:hasGallai}}\label{fig:4-1}
	\end{figure}
	Without loss of generality, assume that $col(xy)=col(zw)=c_1$ and $col(xw)=col(zy)=c_2$. Since $d_{G^{c_1}}(x)\leq n-2k$ and $d_{G^{c_2}}(y)\leq n-2k$, we have 
	$$|V_0|+(|V_i|-1)+1\leq n-2k \text{~for~} i=1,2.$$
	Thus, $$2\leq |V_i|\leq n-2k-1 \text{~for~} i=1,2.$$
	Let $H=G-\{x,y,z,w\}$. We will show that $\Delta^{mon}(H)\leq n-2k-2$. 
	
	For each vertex $v_1\in V_1\setminus \{x,z\}$, by the partition, we know that $col(v_1,G-v_1)\subseteq\{c_1,c_2\}$,  $d_{H^{c_1}}(v_1)\leq d_{G^{c_1}}(v_1)-2\leq n-2k-2$ and $d_{H^{c_2}}(v_1)\leq |V_2|-2< n-2k-2$. Similarly, for each vertex $v_2\in V_2\setminus \{y,w\}$, we have $col(v_2,G-v_2)\subseteq\{c_1,c_2\}$,  $d_{H^{c_2}}(v_2)\leq d_{G^{c_2}}(v_2)-2\leq n-2k-2$ and $d_{H^{c_1}}(v_2)\leq |V_1|-2< n-2k-2$. For each vertex $u\in V_0$, we have $d_{H^{c_1}}(u)\leq d_{G^{c_1}}(u)-2\leq n-2k-2$,  $d_{H^{c_2}}(u)\leq d_{G^{c_2}}(u)-2\leq n-2k-2$ and $d_{H^{c}}(u)\leq |V_0|-1\leq n-2k-|V_1|-1<n-2k-2$ for each color $c\in col(G)\setminus \{c_1,c_2\}$. This implies that $\Delta^{mon}(H)\leq n-2k-2=|V(H)|-2(k-1)$. Recall that $(G,k)$ is a counterexample with $k$ as small as possible, and that the vertex $v_0$ is not contained in any short PC cycle in $H$. This implies that $H$ contains $k-1$ disjoint short PC cycles. Together with the PC cycle $xyzwz$, there exist $k$ disjoint short PC cycles in $G$, a contradiction. This completes the proof of Theorem \ref{thm:hasGallai}.\qed
	
	~\\
	{\bf Proof of Theorem \ref{thm:PCdi2k3}}
	
	By contradiction. Let $(G,k)$ be a counterexample to Theorem \ref{thm:PCdi2k3} with $k$ as small as possible. By Observation \ref{obs:2}, $k\geq 2$. By Lemma \ref{lem:no monochromatic} and Theorem \ref{thm:hasGallai}, $G$ contains no monochromatic edge-cut, admits no Gallai partition, and contains a vertex $v_0$ such that $d^c(v_0)\leq 3$ and $v_0$ is not contained in any short PC cycle. By Lemma \ref{lem:seperation}, there is a partition $V_0,V_1,\ldots, V_p$ of $G$ satisfying Lemma \ref{lem:seperation} $(a)$, $(b)$, $(c)$, and $(d)$, and there exists a rainbow triangle $xyzx$ in $G$ such that $V_x,V_y,V_z$ are distinct sets with $V_0\not\in \{V_x,V_y,V_z\}$. So we have $3\leq p\leq d^c(v_0)\leq 3$.  This forces that $p=3$ (see Figure \ref{fig:5-1}). Without loss of generality, assume that 
	$$x\in V_1,y\in V_2,z\in V_3,$$
	and $$col(xy)=c_1,~col(yz)=c_2,~col(zx)=c_3.$$
	Since colors $c_1$, $c_2$ and $c_3$ appear at most $n-2k$ times at $x$, $y$ and $z$, respectively, we have 
	$$|V_0|+(|V_i|-1)+1\leq n-2k \text{~for~} i=1,2,3.$$
	Thus
	$$1\leq |V_i|\leq n-2k-1 \text{~for~} i=1,2,3.$$
	\begin{figure}[h]
		\centering
		\includegraphics[width=0.32\linewidth]{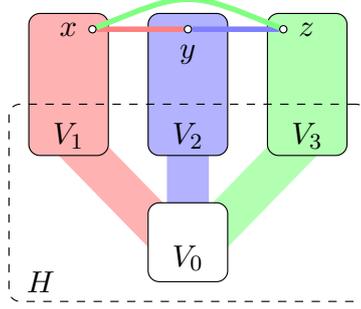}
		\caption{$G$ in the proof of Theorem \ref{thm:PCdi2k3}}\label{fig:5-1}
	\end{figure}
	Let $H=G-\{x,y,z\}$. We will show that $\Delta^{mon}(H)\leq n-2k-1$. 
	
	For each vertex $v_1\in V_1\setminus \{x\}$, by the partition, we know that $col(v_1, G-v_1)\subseteq\{c_1,c_2,c_3\}$,  $d_{H^{c_1}}(v_1)\leq d_{G^{c_1}}(v_1)-1\leq n-2k-1$, and $d_{H^{c_i}}(v_1)\leq |V_i|-1<n-2k-1$ for $i=2,3$. Similarly, for vertices $v_2\in V_2\setminus \{y\}$ and $v_3\in V_3\setminus \{z\}$, we have $col(v_2,G-v_2)\cup col(v_3,G-v_3)=\{c_1,c_2,c_3\}$ and $d_{H^{c_i}}(v_j)\leq n-2k-1$ for $i=1,2,3$ and $j=2,3$.
	For each vertex $u\in V_0$, we have $d_{H^{c_i}}(u)\leq d_{G^{c_i}}(u)-1\leq n-2k-1$ for $i=1,2,3$, and $d_{H^c}(u)\leq |V_0|-1<n-2k-1$ for each color $c\in col(G)\setminus \{c_1,c_2,c_3\}$. 
	This implies that $\Delta^{mon}(H)\leq n-2k-1=|V(H)|-2(k-1)$. By the choice of $(G,k)$, and since $v_0$ is not contained in any short PC cycle of $H$, we conclude that $H$ contains $k-1$ disjoint short PC cycles. Together with the PC cycle $xyzwz$, there exist $k$ disjoint short PC cycles in $G$, a contradiction. This completes the proof of Theorem \ref{thm:PCdi2k3}.\qed
	
	~\\
	{\bf Proof of Theorem \ref{thm:chracterization}}
	
	By Observation \ref{obs:2} and Theorem \ref{thm:k=2}, we have $k\geq 3$. If $G$ contains a rainbow triangle $xyzx$, then one easily checks that $G-\{x,y,z\}$ is a smaller counterexample to Conjecture \ref{con:disjoint}, a contradiction. So, $G$ contains no rainbow triangles, and thus has a Gallai partition $U_1,U_2,\dots, U_q$
	by Lemma~\ref{lem:Gallai}. By Lemma \ref{lem:no monochromatic}, $G$ contains no monochromatic edge-cut. So $q\geq 4$, and we can assume that the two colors appearing between the parts are $red$ and $blue$. Since $\Delta^{mon}(G)\leq n-3k+1$, by the definition of Gallai partition, we have $|U_i|\leq n-3k+1$ for all $i$ with $1\leq i\leq q$. This implies that $\sum_{c\in col(G),c\neq red, blue} d_{G^c}(v_i)\leq |U_i|-1<n-3k+1$ for all $i$ with $1\leq i\leq q$. If $|col(G)|\neq 2$ or $3$, then $|col(G)|\geq 4$. In this case, let $H$ be a colored graph obtained from $G$ by recoloring all the edges which are neither $red$ nor $blue$ in $G$ with the color $green$. Clearly, $H$ contains no $k$ disjoint short PC cycles (otherwise, there exist $k$ disjoint short PC cycles in $G$), and $\Delta^{mon}(H)\leq |V(H)|-3k+1$. So, $(H,k)$ is a counterexample to Conjecture \ref{con:disjoint} with $|V(H)|=|V(G)|$ and $|col(H)|<|col(G)|$, a contradiction. Hence, we conclude that $|col(G)|=2$ or $3$. This completes the proof of $(a)$, $(b)$, $(c)$, and $(d)$ of Theorem \ref{thm:chracterization}.

	Finally, we prove $(e)$ of Theorem \ref{thm:chracterization}, by contradiction. Suppose to the contrary, that there exists a set $S\subseteq V(G)$ with $|S|\leq k-1$, and a vertex $v_0\in V(G)\setminus S$ such that $v_0$ is not contained in any short PC cycle in $G-S$. Let $H=G-S$. Then,
	$$\Delta^{mon}(H)\leq \Delta^{mon}(G)=n-3k+1=(n-|S|)-2k+(|S|-k+1)\leq |V(H)|-2k.$$ 
	By Theorem \ref{thm:hasGallai} and the fact that $v_0$ is not contained in any short PC cycle in $H$, the colored complete graph $H$ contains $k$ disjoint short PC cycles, which are also contained in $G$, a contradiction. This completes the proof of Theorem \ref{thm:chracterization}.\qed
	
	\setcounter{claim}{0}
	\section{Proof of Theorem \ref{thm:equivalent}}
	\begin{proof}\label{sec:eq}
		We first prove that Proposition~\ref{pro:dicycle} implies Proposition~\ref{pro:PCcycle}.
		Assume that Proposition~\ref{pro:dicycle} is true. Let $G$ be a colored $K_n$ satisfying $\Delta^{mon}(G)\leq n-f(k)-1$ and containing no PC cycles of length $i$ for any $i\in I$. If $G$ contains a monochromatic edge-cut, then by Lemma \ref{lem:no monochromatic}, there exist $k$ disjoint PC cycles in $G$, and we are done. If each vertex in $G$ with color degree at most $\ell$ is contained in some PC cycle, there is nothing to prove. So, we assume that there is no monochomatic edge-cut in $G$, and that there exists a vertex $v_0\in V(G)$ with $2\leq d^c(v_0)\leq \ell$ that is not contained in any PC cycle.  By Lemma \ref{lem:seperation}, there exists a partition $V_0,V_1,V_2,\ldots, V_p$ of $V(G)$ satisfying Lemma \ref{lem:seperation} $(a)$, $(b)$, $(c)$, and $(d)$. Now, we define a $p$-partite tournament $MT$ with vertex set $V(MT)=V_1\cup V_2\cup\ldots \cup V_p$ and arc set
		$$A(MT)=\{xy: V_x\neq V_y, \; col(xy)=col(v_0y)\}.$$
		Note that for each vertex $v_i\in V_i~(1\leq i\leq p)$, there are at least $f(k)$ vertices joined to $v_i$ by edges with colors different from $c_i$. This implies that $\delta^+(MT)\geq f(k)$. Suppose that $MT$ contains a dicycle $C$ of length $s$ for some $s\in I$. Then, by the definition of $MT$, $C$ corresponds to a PC cycle of length $s$ in $G$, a contradiction. So, $MT$ contains no dicycles of length $i$ for any $i\in I$.
		Now, we define an $\ell$-partite tournament $MT'$, as follows. First note that, by Lemma \ref{lem:seperation} (a), $p\leq d^c(v_0)\leq \ell$. If $p=\ell$, we take $MT'=MT$; otherwise, let $$V(MT')=V(MT)\cup \{u_1,u_2,\ldots,u_{\ell-p}\}$$ and $$A(MT')=A(MT)\cup \{u_ix: x\in V(MT), \; 1\leq i\leq \ell-p\}.$$ Then, $MT'$ is an $\ell$-partite tournament with $\delta^+(MT')\geq f(k)$, and none of $\{u_1,u_2,\ldots,u_{\ell-p}\}$ is contained in a dicycle in $MT'$. So, $MT'$ contains no dicycles of  length $i$ for any $i\in I$. By Proposition \ref{pro:dicycle},  $MT'$ contains $k$ disjoint dicycles, which are contained in $MT$ and correspond to $k$ disjoint PC cycles in $G$. This completes the proof of the first implication.
		
		Next, we prove that Proposition~\ref{pro:PCcycle} implies Proposition~\ref{pro:dicycle}.
		Assume that Proposition~\ref{pro:PCcycle} is true. Let $MT$, with vertex set partitioned as $V_1\cup V_2\cup\ldots \cup V_{\ell}$, be an $\ell$-partite tournament satisfying $\delta^+(MT)\geq f(k)$ and containing no dicycles of length $i$ for all $i\in I$. 
		We define a colored complete graph $G$ with 
		$$V(G)=\bigcup _{1\leq i\leq \ell}V_i\cup \{v_0\},$$
		for $1\leq i\leq j\leq \ell$, $v_i\neq v_j$, $v_i\in V_i, v_j\in V_j$,
		$$col(v_0v_i)=c_i$$
		and
		\[
		col(v_iv_j)=
		\begin{cases}
		c_j,&\text{ if $v_iv_j\in A(MT)$;}\cr
		c_i,&\text{otherwise.}\cr
		\end{cases}
		\]
		Let $n=|V(G)|$. For a vertex $v_i\in V_i~(1\leq i\leq \ell)$, denote by $N^+_{MT}(v_i)$ the set of out-neighbors of $v_i$ in $MT$. Since $|N^+_{MT}(v_i)|=d^+_{MT}(v_i)\geq f(k)$ and $N^+_{MT}(v_i)\cap V_i=\emptyset$, we have $|V_i|\leq n-f(k)-1$.  Note that each vertex in $N^+_{MT}(v_i)$ is joined to $v_i$ by an edge with color distinct from $c_i$ in $G$. So, the color $c_i$ appears at most $n-f(k)-1$ times at $v_i$, and any color $c_j~(j\neq i)$ may appear at most $|V_j|\leq n-f(k)-1$ times at $v_i$. For the vertex $v_0$, each color appears at most $|V_j|=n-f(k)-1$ times at $v_0$. Thus, we have $\Delta^{mon}(G)\leq n-f(k)-1$.
		\begin{claim}\label{clm:notin}
			~\\	
			$(a)$ $v_0$ is not contained in any PC cycle in $G$.\\
			$(b)$ each edge $xy$ is not contained in any PC cycle in $G$ for $x,y\in V_i~(1\leq i\leq \ell)$.
		\end{claim}
		\begin{proof}
			Suppose to the contrary of $(a)$, that $C$ is a PC cycle in $G$ containing $v_0$. Orient the edges of $C$ in one of the two directions along $C$. Choose a vertex $u\in V(C)\setminus\{v_0\}$, and assume that $u\in V_i$ for some $i$ with $1\leq i\leq \ell$. Then, we obtain $col(u^-u)=c_i$ and $col(u^+u)=c_i$, by following the paths $v_0\overrightarrow{C}u^-u$ and $v_0\overleftarrow{C}u^+u$, respectively. (Here, $u^+$ and $u^-$ denote the immediate successor and predecessor of $u$ on $C$ in the direction specified by the orientation of $C$, respectively, and $\overrightarrow{C}$ and $\overleftarrow{C}$ denote the traversal of $C$ in the direction of the orientation, and in the opposite direction, respectively.) Thus $col(u^+u)=col(u^-u)$, a contradiction. Similarly, we can prove that $xy$ is not contained in any PC cycles for $x,y\in V_i~(1\leq i\leq \ell)$.
		\end{proof}
		Claim \ref{clm:notin} implies that each PC cycle in $G$ corresponds to a dicycle in $MT$. Hence, $G$ contains no PC cycles of length $i$ for any $i\in I$. By Proposition \ref{pro:PCcycle}  and Claim \ref{clm:notin} $(a)$, $G$ contains $k$ disjoint PC cycles, which correspond to $k$ disjoint dicycles in $MT$.
		
		This completes the proof.	
	\end{proof}


\begin{thebibliography}{10}
		
		\bibitem{Bai-Li: 2015}
		{Y. Bai, B. Li and H. Li},
		\newblock Vertex-disjoint cycles in bipartite tournaments,
		\newblock{\em Discrete Math.,} {\bf 338} (2015) 1307--1309. 
		
		\bibitem{PC} Y. Bai and B. Li, Private communication (2016).
		
		\bibitem{Bang-Jensen: 2014}
		{J. Bang-Jensen, S. Bessy and S. Thomass\'{e}},
		\newblock {Disjoint 3-cycles in tournaments: a proof of the Bermond-Thomassen Conjecture for tournaments},
		\newblock {\em J. Graph Theory,} {\bf 75} (2014) 284--302.		
		
		\bibitem{Bang-Jensen-Gutin: 2009}
		{J. Bang-Jensen and G. Gutin},
		\newblock {\em Digraphs: Theory, Algorithms and Applications}, Second edition,
		\newblock Springer Monographs in Mathematics, Springer-Verlag London Ltd., London, 2009.
		
		\bibitem{Bermond-Thomassen: 1981}
		{J.-C. Bermond and C. Thomassen},
		\newblock {Cycles in digraphs -- a survey},
		\newblock {\em J. Graph Theory,} {\bf 5} (1981) 1--43.
		
		
		\bibitem{Bondy:2008}
		{J.A. Bondy and U.S.R. Murty},
		\newblock {\em Graph Theory},
		\newblock {Springer Graduate Texts in Mathematics, vol. 244 (2008).}
		
		\bibitem{Fujita: 2011}
		{S. Fujita and C. Magnant},
		\newblock Properly colored paths and cycles,
		\newblock{\em Discrete Appl. Math.,} {\bf 159} (2011) 1391--1397.
		
		
		\bibitem{Gallai: 1967}
		{T. Gallai},
		\newblock Transitiv orientierbare Graphen,
		\newblock{\em Acta Math. Hungar.}, {\bf 18} (1967) 25--66.
		
		\bibitem{Grossman:1983}
		{J.W. Grossman and R. H\"{a}ggkvist},
		\newblock {Alternating cycles in edge-partitioned graphs},
		\newblock {\em J. Combin. Theory Ser. B}, {\bf 34} (1983) 77-81.
		
		\bibitem{X.Li: 2008}
		{M. Kano and X. Li}, 
		\newblock {Monochromatic and heterochromatic subgraphs in edge-colored graphs - a survey}, \newblock {\em Graphs Combin.}, {\bf 24} (2008) 237--263.
		
		\bibitem{B.Li: 2014}
		{B. Li, B. Ning, C. Xu and S. Zhang},
		\newblock Rainbow triangles in edge-colored graphs,
		\newblock {\em European J. Combin.,} {\bf 36} (2014) 453--459.
		
		\bibitem{Lo: 2014-2}
		{A. Lo},
		\newblock An edge-colored version of Dirac's theorem,
		\newblock{\em SIAM J. Discrete Math.,} {\bf 28} (2014) 18--36.
		
		\bibitem{Yeo: 1997}
		{A. Yeo},
		\newblock A note on alternating cycles in edge-colored graphs,
		\newblock{\em J. Combin. Theory Ser. B,} {\bf 69} (1997) 222--225.
		
	\end{thebibliography}
\end{document}